\documentclass[12pt]{amsart}
\usepackage{amsmath,amsthm,amsfonts,amssymb,latexsym}

\usepackage{hyperref}

\usepackage{enumerate}
\usepackage[shortlabels]{enumitem}

\begin{document}

\theoremstyle{plain}

\newtheorem{thm}{Theorem}[section]

\newtheorem{lem}[thm]{Lemma}
\newtheorem{Problem B}[thm]{Problem B}

\newtheorem{pro}[thm]{Proposition}
\newtheorem{cor}[thm]{Corollary}
\newtheorem{que}[thm]{Question}
\newtheorem{rem}[thm]{Remark}
\newtheorem{defi}[thm]{Definition}

\newtheorem*{thmA}{Theorem A}
\newtheorem*{thmB}{Theorem B}
\newtheorem*{corB}{Corollary B}
\newtheorem*{thmC}{Theorem  C}
\newtheorem*{thmD}{Theorem D}
\newtheorem*{thmE}{Theorem E}

\newtheorem*{thmAcl}{Main Theorem$^{*}$}
\newtheorem*{thmBcl}{Theorem B$^{*}$}

\newcommand{\Maxn}{\operatorname{Max_{\textbf{N}}}}
\newcommand{\Syl}{\operatorname{Syl}}
\newcommand{\dl}{\operatorname{dl}}
\newcommand{\Con}{\operatorname{Con}}
\newcommand{\cl}{\operatorname{cl}}
\newcommand{\Stab}{\operatorname{Stab}}
\newcommand{\Aut}{\operatorname{Aut}}
\newcommand{\Sym}{\operatorname{Aut}}
\newcommand{\Ker}{\operatorname{Ker}}
\newcommand{\fl}{\operatorname{fl}}
\newcommand{\Irr}{\operatorname{Irr}}
\newcommand{\SL}{\operatorname{SL}}
\newcommand{\FF}{\mathbb{F}}
\newcommand{\NN}{\mathbb{N}}
\newcommand{\N}{\mathbf{N}}
\newcommand{\bfC}{\mathbf{C}}
\newcommand{\bfO}{\mathbf{O}}
\newcommand{\bfF}{\mathbf{F}}
\newcommand{\OO}{\mathbf{O}}

\renewcommand{\labelenumi}{\upshape (\roman{enumi})}

\newcommand{\PSL}{\operatorname{PSL}}
\newcommand{\PSU}{\operatorname{PSU}}
\newcommand{\alt}{\operatorname{Alt}}
\newcommand{\GL}{\operatorname{GL}}
\newcommand{\SO}{\operatorname{SO}}
\newcommand{\SU}{\operatorname{SU}}

\providecommand{\V}{\mathrm{V}}
\providecommand{\E}{\mathrm{E}}
\providecommand{\ir}{\mathrm{Irr_{rv}}}
\providecommand{\Irrr}{\mathrm{Irr_{rv}}}
\providecommand{\re}{\mathrm{Re}}

\numberwithin{equation}{section}
\def\irrp#1{{\rm Irr}_{p'}(#1)}

\def\ibrrp#1{{\rm IBr}_{\Bbb R, p'}(#1)}
\def\Z{{\mathbb Z}}
\def\C{{\mathbb C}}
\def\Q{{\mathbb Q}}
\def\irr#1{{\rm Irr}(#1)}
\def\aut#1{{\rm Aut}}
\def\ibr#1{{\rm IBr}(#1)}
\def\irrp#1{{\rm Irr}_{p^\prime}(#1)}
\def\irrq#1{{\rm Irr}_{q^\prime}(#1)}
\def \c#1{{\cal #1}}
\def\cent#1#2{{\bf C}_{#1}(#2)}
\def\syl#1#2{{\rm Syl}_#1(#2)}
\def\nor{\trianglelefteq\,}
\def\oh#1#2{{\bf O}_{#1}(#2)}
\def\Oh#1#2{{\bf O}^{#1}(#2)}
\def\zent#1{{\bf Z}(#1)}
\def\det#1{{\rm det}(#1)}
\def\ker#1{{\rm ker}(#1)}
\def\norm#1#2{{\bf N}_{#1}(#2)}
\def\alt#1{{\rm Alt}(#1)}
\def\iitem#1{\goodbreak\par\noindent{\bf #1}}
   \def \mod#1{\, {\rm mod} \, #1 \, }
\def\sbs{\subseteq}

\def\gc{{\bf GC}}
\def\ngc{{non-{\bf GC}}}
\def\ngcs{{non-{\bf GC}$^*$}}
\newcommand{\notd}{{\!\not{|}}}
\newcommand{\Out}{{\mathrm {Out}}}
\newcommand{\Mult}{{\mathrm {Mult}}}
\newcommand{\Inn}{{\mathrm {Inn}}}
\newcommand{\IBR}{{\mathrm {IBr}}}
\newcommand{\IBRL}{{\mathrm {IBr}}_{\ell}}
\newcommand{\IBRP}{{\mathrm {IBr}}_{p}}
\newcommand{\ord}{{\mathrm {ord}}}
\def\id{\mathop{\mathrm{ id}}\nolimits}
\renewcommand{\Im}{{\mathrm {Im}}}
\newcommand{\Ind}{{\mathrm {Ind}}}
\newcommand{\diag}{{\mathrm {diag}}}
\newcommand{\soc}{{\mathrm {soc}}}
\newcommand{\End}{{\mathrm {End}}}
\newcommand{\sol}{{\mathrm {sol}}}
\newcommand{\Hom}{{\mathrm {Hom}}}
\newcommand{\Mor}{{\mathrm {Mor}}}
\newcommand{\Mat}{{\mathrm {Mat}}}
\def\rank{\mathop{\mathrm{ rank}}\nolimits}
\newcommand{\Tr}{{\mathrm {Tr}}}
\newcommand{\tr}{{\mathrm {tr}}}
\newcommand{\Gal}{{\it Gal}}
\newcommand{\Spec}{{\mathrm {Spec}}}
\newcommand{\ad}{{\mathrm {ad}}}

\newcommand{\Char}{{\mathrm {char}}}
\newcommand{\pr}{{\mathrm {pr}}}
\newcommand{\rad}{{\mathrm {rad}}}
\newcommand{\abel}{{\mathrm {abel}}}
\newcommand{\codim}{{\mathrm {codim}}}
\newcommand{\ind}{{\mathrm {ind}}}
\newcommand{\Res}{{\mathrm {Res}}}
\newcommand{\Lie}{{\mathrm {Lie}}}
\newcommand{\Ext}{{\mathrm {Ext}}}
\newcommand{\Alt}{{\mathrm {Alt}}}
\newcommand{\AAA}{{\sf A}}
\newcommand{\SSS}{{\sf S}}
\newcommand{\CC}{{\mathbb C}}
\newcommand{\CB}{{\mathbf C}}
\newcommand{\RR}{{\mathbb R}}
\newcommand{\QQ}{{\mathbb Q}}
\newcommand{\ZZ}{{\mathbb Z}}
\newcommand{\bfN}{{\mathbf N}}
\newcommand{\bfZ}{{\mathbf Z}}
\newcommand{\EE}{{\mathbb E}}
\newcommand{\PP}{{\mathbb P}}
\newcommand{\cG}{{\mathcal G}}
\newcommand{\cH}{{\mathcal H}}
\newcommand{\cQ}{{\mathcal Q}}
\newcommand{\GA}{{\mathfrak G}}
\newcommand{\cT}{{\mathcal T}}
\newcommand{\cS}{{\mathcal S}}
\newcommand{\cR}{{\mathcal R}}
\newcommand{\GCD}{\GC^{*}}
\newcommand{\TCD}{\TC^{*}}
\newcommand{\FD}{F^{*}}
\newcommand{\GD}{G^{*}}
\newcommand{\HD}{H^{*}}
\newcommand{\GCF}{\GC^{F}}
\newcommand{\TCF}{\TC^{F}}
\newcommand{\PCF}{\PC^{F}}
\newcommand{\GCDF}{(\GC^{*})^{F^{*}}}
\newcommand{\RGTT}{R^{\GC}_{\TC}(\theta)}
\newcommand{\RGTA}{R^{\GC}_{\TC}(1)}
\newcommand{\Om}{\Omega}
\newcommand{\eps}{\epsilon}
\newcommand{\al}{\alpha}
\newcommand{\chis}{\chi_{s}}
\newcommand{\sigmad}{\sigma^{*}}
\newcommand{\PA}{\boldsymbol{\alpha}}
\newcommand{\gam}{\gamma}
\newcommand{\lam}{\lambda}
\newcommand{\la}{\langle}
\newcommand{\ra}{\rangle}
\newcommand{\hs}{\hat{s}}
\newcommand{\htt}{\hat{t}}
\newcommand{\tn}{\hspace{0.5mm}^{t}\hspace*{-0.2mm}}
\newcommand{\ta}{\hspace{0.5mm}^{2}\hspace*{-0.2mm}}
\newcommand{\tb}{\hspace{0.5mm}^{3}\hspace*{-0.2mm}}
\def\skipa{\vspace{-1.5mm} & \vspace{-1.5mm} & \vspace{-1.5mm}\\}
\newcommand{\tw}[1]{{}^#1\!}
\renewcommand{\mod}{\bmod \,}

\marginparsep-0.5cm

\renewcommand{\thefootnote}{\fnsymbol{footnote}}
\footnotesep6.5pt

\newcommand{\sym}{{\mathrm {Sym}}}
\newcommand{\Sp}{{\mathrm {Sp}}}
\newcommand{\PSp}{{\mathrm {PSp}}}

\title{Modular Characters, Hall subgroups,  and Normal Complements}

\author{Robert Guralnick}
\address{Department of Mathematics, University of Southern California, 3620 S. Vermont Ave., Los Angeles, CA 90089, USA}
\email{guralnic@usc.edu}
\author{Gabriel Navarro}
\address{Departament de Matem\`atiques, Universitat de Val\`encia, 46100 Burjassot,
Val\`encia, Spain}
\email{gabriel@uv.es}
  
\thanks{The first  author
gratefully acknowledges the support of the NSF grant DMS-1901595 and he thanks A. Kleshchev
for a discussion of representations of the symmetric group.  The second author
is supported by MCIN/AEI 10.13039/501100011033,
  Grants PID2022-137612NB-I00 and ERDF ``A way of making Europe", and he thanks A. Moret\'o for conversations on this subject.
  Both authors thank G. Malle for comments and corrections, and H. Tong-Viet for pointing out \cite{F}.}



\begin{abstract}
If $H$ is a Hall subgroup of a finite group $G$, it was proven in 1989 using the classification of finite simple groups that all the irreducible complex characters of $H$ extend to $G$ if
and only if there is $N\nor G$ such that $HN=G$ and $H\cap N=1$.  In this note we offer a modular version of this result under more general hypothesis.
\end{abstract}

\maketitle

\section{Introduction}
One of the first applications of character theory was to produce normal subgroups of finite groups. In many cases, the   problem   was: given a Hall subgroup $H$ of a finite group $G$,  when $H$ does  have a normal complement in $G$? In other words, if $H$ is a subgroup of $G$ such that $|H|$ and $|G:H|$ are coprime, when there is $N \nor G$ such that $G=NH$ and 
$H\cap N=1$?  Originally, this problem was treated extensively in the literature, for instance by  Baer, Frobenius, Higman, Sah, Thompson and Wielandt among many others
(\cite{B}, \cite{Fr}, \cite{Hi}, \cite{S}, \cite{T}, \cite{W}).  If $H$ is a Sylow $p$-subgroup of $G$
(or more generally, if $H$ is nilpotent), the famous Brauer--Suzuki theorem asserts that this happens if and only if $h^G \cap H=h^H$ for every $h \in H$ (where $h^G$ is the conjugacy class of $h$ in $G$).
 If $H$ is not nilpotent, this is not necessarily true, as shown, for instance, by $H={\sf S}_{p-1}$ and $G={\sf S}_p$ for $p>4$
prime.

If the attention is turned to complex characters, it is trivial to observe that
if $H$ has a normal complement in $G$, then every irreducible
 character $\alpha \in \irr H$ extends to some $\beta \in \irr G$. Of course, it was natural to ask about the converse, and this question was settled in 1989 by P. Ferguson, using the Classification of Finite Simple Groups (CFSG).
 
  At first sight,  this result does not seem to admit a  modular version with $p$-Brauer characters, if only by considering the case where $H \in \syl pG$.  Less trivially, for $p=3$, notice that all irreducible $3$-Brauer characters of ${\sf A}_4$ extend to ${\sf A}_5$.  
  In fact,  note that if $G(q)$ is a finite simply connected
  simple group of Lie type over the field of $q$ elements, then all its irreducible representations in the defining characteristic 
   extend even to the simple algebraic group.

\medskip

 The main result in this note is to prove the following.    
  \begin{thmA}
  Let $G$ be a finite group,  let $H$ be a Hall subgroup of $G$, and let $p$
  be a prime not dividing $|H|$.
Then every $\alpha \in \irr H$ extends to a $p$-Brauer character of $G$
if and only if $H$ has a normal complement in $G$.
  \end{thmA}

   The easier case where $H$ is solvable is proven in Section 2 below, without using the CFSG.
   In the case where $H$ is not solvable, our new  strategy here is not to assume that
  every character of $H$ extends to $G$, but only that every {\sl primitive} character does, together with the  usual group theoretical conditions
  that follow from the character extension property.  Then CFSG is used.   
  \medskip

  Of course, by choosing any prime $p$ not dividing $|G|$, we see that Theorem A implies the main result of \cite{F}.
  
  We prove a stronger result (for ordinary characters) if  $G$ is an almost simple group  where we replace the Hall subgroup 
  assumption by the condition that $H$ has odd index in $G$.    See Theorem \ref{ccl2}.
  
  \section{Solvable Hall subgroups}
  Our notation for characters is as in \cite{I}. Our notation for Brauer characters is as in \cite{N1}. Hence, we have chosen  a fixed maximal ideal $M$ of the ring of algebraic integers $\bf R$ containing $p$, 
  worked in the algebraically closed field $F={\bf R}/M$, and we have defined a set $\ibr G$ of the irreducible
  $p$-Brauer characters of every finite group $G$, for the rest of the paper.
  
  \medskip
   In this section, we first dispose of the easy case of Theorem A where
  $H$ is solvable.
  
  \medskip
  
   If $H \le G$ and every irreducible character of $H$ extends to $G$,  the following two are essentially the only known group theoretical consequences at our disposal.   If $G$ is a finite group, in this paper $g^G$ denotes the conjugacy class of $g$ in $G$.
  Recall that a subgroup $H$ of $G$ is $c$-closed if $h^G \cap H=h^H$ for every $h \in H$.

  \begin{lem}\label{elem}
  Let $H \le G$, $p$ a prime such that $p$ does not divide $|H|$.
  Suppose that every character of $H$ extends to a $p$-Brauer character of $G$.
  Then
  \begin{enumerate}[(i)]
  \item
  For every $h \in H$ we have that $h^G \cap H=h^H$.
  \item
  For every $L \nor H$, there is $K \nor G$ such that $K\cap H=L$, and such that every character of $HK/K$ extends to a $p$-Brauer character of $G/K$.
  \end{enumerate}
  \end{lem}
  
  \begin{proof}
  Suppose that $h^g \in H$ for some $g \in G$. Let $\alpha \in \irr H$ and let $\hat\alpha \in \ibr G$ be such that $\hat\alpha_H=\alpha$. Then
  $\alpha(h^g)=\hat\alpha(h^g)=\hat\alpha(h)=\alpha(h)$, and (a) follows from
  the orthogonality of characters. 
  
  For every $\alpha \in \irr{H/L}$, choose $\hat\alpha \in \ibr G$ such that
  $\hat\alpha_H=\alpha$. 
  Let 
  $$K=\bigcap_{\alpha \in \irr{H/L}} \ker{\hat\alpha} \nor G \, .$$
  Notice that 
  $$K \cap H=\bigcap_{\alpha \in \irr{H/L}} \ker{\alpha}=L\, .$$
  Suppose that $\gamma \in \irr{HK/K}$. Then $\gamma_H=\alpha \in \irr{H/L}$. Since $K \sbs \ker{\hat\alpha}$, we have that $\hat\alpha_{KH}=\gamma$. 
\end{proof}

\begin{lem}\label{nor}
Suppose that $H$ is a Hall subgroup of $G$ which is $c$-closed in $G$.
Let $N\nor G$ such that $N \sbs H$ and $H/N$ has a normal complement $K/N$ in $G/N$. Then $H$ has a normal complement in $G$.
\end{lem}
\begin{proof} Let $\pi$ be the set of primes dividing $|H|$. 
 Now let $n \in N$ and $k \in K$. Then $n^G \cap H=n^H$ by hypothesis.
 Thus $n^k=n^h$ for some $h \in H$, and thus $n^{kh^{-1}}=n$
 and thus $K \sbs \cent Gn H$. Hence $G=KH \sbs \cent Gn H$ and
 we have that $|G:\cent Gn|$ is a $\pi$-number. Hence so  is $|K:\cent Kn|$. Since $N$
 is a Hall $\pi$-subgroup of $K$, we conclude that
 $K=N\cent Kn$ for every $n \in N$.  Hence all $N$-conjugacy classes are $K$-invariant. Let $A$ be a complement of $N$ in $K$. By coprime action we conclude that $A \nor K$. (Apply for instance, Lemma 13.10 of \cite{I}.) Hence $A \nor G$
 since $A$ is characteristic in $K \nor G$. 
\end{proof}
\begin{thm}\label{solv}
 Let $G$ be a finite group,  let $H$ be a Hall subgroup of $G$, and let $p$
  be a prime not dividing $|H|$.
  Assume that $H$ is solvable.
If every $\alpha \in \irr H$ extends to a $p$-Brauer character of $G$, 
then  $H$ has a normal complement in $G$.
\end{thm}

\begin{proof}
We argue by induction on $|G|$. By Lemma \ref{elem}(i),
we have that $H$ is $c$-closed in $G$. 
Let $M$ be a minimal normal subgroup of $H$, so that $M$ is
an abelian $q$-group for some prime $q$. By Lemma \ref{elem},
 let $K \nor G$ such that $K\cap H=M$ and
such that all irreducible characters of $HK/K$ extend to a $p$-Brauer character of $G/K$. Since $K>1$, by induction, there is $L \nor G$ such that $HL=G$ and $H\cap L=M$.
If $M \nor G$, then we apply Lemma \ref{nor}, and we are done.
So we may assume that $\norm GM< G$. By induction, we have that
$\norm GM=HX$ where $X \nor \norm GM$ and $H\cap X=1$.
Now, $\norm LM=M \times X$. Since $H$ is a Hall subgroup, we have that
$M$ is a Sylow $q$-subgroup of $L$. Also $M \sbs \zent{\norm LM}$.
Therefore $M$ has a normal $q$-complement $Y$ in $L$ by Burnside's Theorem. Then $Y\nor G$ is a complement of $H$ in $G$. 
\end{proof}

   \section{Primitive Characters}

   Recall that an irreducible complex character $\chi \in \irr G$ is primitive if there
   is no $U<G$ and $\gamma \in \irr U$  such that $\gamma^G=\chi$. 
   
 \begin{lem} \label{prim}
  Let $G$ be a finite group,  let $H$ be a Hall subgroup of $G$,
 let $N$ be  a normal  subgroup of $G$ such that $H\cap N=1$.
 Let $p$ be a prime not dividing $|H|$.
  Let $\alpha \in \irr H$ be primitive.
 If $\alpha$ extends to some $p$-Brauer character $\beta$ of $G$,  
 then $\beta_N$ is $\beta(1)\nu$ for some linear  character $\nu \in \ibr N$.
 Moreover, there is an extension $\gamma \in \ibr G$
 of $\alpha$ that contains $N$ in its kernel.
 \end{lem}

 \begin{proof}
 Let $\beta \in \ibr G$ such that $\beta_H=\alpha$. 
 We claim that $\beta_N=\beta(1) \nu$, for some linear $\nu \in \ibr N$. 
 Let $\mathcal X$ be an $F$-representation affording $\beta$.
 Let $\zent{\beta}$ be the subgroup $\{g \in G \mid \mathcal X(g)=\lambda_g I \}$, where $\lambda_g \in F$ is a scalar. 
  For each prime $q \ne p$ dividing $|N|$,
  by coprime action, there exists   an $H$-invariant Sylow $q$-subgroup $Q$ of $N$. Let $U=HQ$.
 Let $\chi=\beta_U \in \irr U$.  Let $\lambda \in \irr Q$ be under $\chi$.
 Then $\lambda(1)$ divides $\chi(1)$, which divides $|H|$. Since $\lambda(1)$ divides $|Q|$, we have that $\lambda$
 is linear. 
   By   the Clifford correspondence
 (Theorem 6.11 of \cite{I}),  there is $\tau \in \irr{U_\lambda}$ lying over $\lambda$  such that
 $$\tau^U=\chi \, .$$
 By Mackey, $(\tau_{H_\lambda})^H=\alpha$. Since $\alpha$ is primitive, then we conclude that 
 $H=H_\lambda$ and that $\lambda$ is $H$-invariant. Therefore $\beta_{Q}=\beta(1)\nu$, for some linear character $\nu \in \irr Q$, and we conclude that
 $Q \sbs \zent{\beta_U} \sbs \zent \beta$.
If $W=\zent \beta \cap N$, it then follows that $N/W$ is a $p$-group
and that $\beta_W=\beta(1)\delta$ for some linear character $\delta \in \ibr W$. By Green's theorem (8.11 of \cite{N1}), we have that $\beta_N=\beta(1)\nu$ for some linear $\nu \in \ibr N$, as claimed.

Next, we claim that $\nu$ extends to $G$.  By Theorem 8.29 of \cite{N1}, it
suffices to show that $\nu$ extends to $R$, where $R/N$ is
a Sylow $r$-subgroup of $G/N$ for all primes $r$. If $r$ divides $|H|$, this follows from Theorem 8.23 of \cite{N1}. Suppose that $r$ does not divide $|H|$.
The $r$ does not divide $\alpha(1)=\beta(1)$. Write $\beta_R=e_1 \delta_1 + \ldots + e_t\delta_r$,
  where $\delta_i \in \irr{R|\nu}$. By Theorem 8.30 of \cite{N1}, then notice that for  some  $i$ we necessarily have that 
  $\delta_i(1)/\nu(1)=1$. This proves the claim. Let $\mu \in \irr G$ be an extension of $\nu$. 
If $\pi$ is the set of primes dividing $|N|$, we may assume that $o(\mu)$ is a $\pi$-number, by considering the $\pi$-part $\mu_\pi$ of $\mu$. In particular,
 $\mu_H=1_H$.  Now, consider $\gamma=\beta\bar \mu  \in \ibr G$, which has $N$ in its kernel and extends $\alpha$.  
 \end{proof}
 
 \medskip
 
 If $G$ is $\pi$-separable, $H$ is a Hall $\pi$-subgroup and $N\nor G$ is a $\pi'$-subgroup, it is also not difficult to prove that
 if $\alpha \in \irr H$ extends to some $p$-Brauer character of $G$, then there is an extension containing $N$ in its kernel. 
 (See the proof of Lemma 2.1 of \cite{MN}.) Using this and  Lemma \ref{nor}, we could easily prove Theorem A in the case
 where $G$ is $\pi$-separable. 
 
 \medskip
 
 We shall use the following extension theorem.
 
 \begin{thm}\label{pext}
 Suppose that $N=S^n$ is a minimal non-abelian normal subgroup of $G$,
 where $S$ is simple,  and let $\tau \in \irr S$. If $\tau$ extends to  
 some $\beta \in \irr{{\rm Aut}(S)}$, then $\gamma=\tau \times \cdots \times \tau$ extends to
 some $\chi \in \irr G$. Furthermore, if $\tau$ is primitive, then $\chi$ is primitive. 
 \end{thm}
 
 \begin{proof}
 See,  for instance,  Lemma 5 of \cite{BCLP}, or Corollary 10.5 of \cite{N2}.
 We have that $\gamma$ is primitive by Theorem 4.5 of \cite{H}.
 By Mackey (see Problems 5.6 and 5.7 of \cite{I}), we have that $\chi$ is primitive. 
 \end{proof}

   We notice the following variation of  Lemma 4.2 of \cite{M}.
  
  \begin{lem} \label{l:primitve}  If $S$ is a finite nonabelian simple, group then there exists a primitive character $\chi \ne 1$ such
  that $\chi$ extends to $\mathrm{Aut}(S)$.
  \end{lem}
  
  \begin{proof}  If $S$ is sporadic, this follows by an easy inspection. In general,  the minimal degree character extends
  and this is primitive, except in the following cases: If $S=M_{12}$, we take the character of degree 45;
  if $S=J_2$, we take the character of degree 63;  if $S=J_3$, then we take the character of degree 324;
  if $S=He$, we take the character of degree 680; if  $S=HN$, we take the character of degree 760. 
  
   If $S={\sf A}_n$, $n \ne 6$, let $\chi$ be the character of degree $n-1$.  If $S={\sf A}_6$, then we take the character of degree 9.
  If $S$ is of Lie type in characteristic $p$,  then it is known
  that the Steinberg character extends to the full automorphism group.  Since the degree is a power of $p$, this character could
  only be imprimitive if $S$ has a subgroup of index a power of $p$.   This only happens for $S=\mathrm{PSL}(3,2)$,  $\mathrm{PSp}(4,3)$ or ${\rm PSL}(2,11)$
  (see \cite[Theorem A]{Se} and \cite[Theorem 1]{G}).   In the last case, the subgroups of index a power of $p$ is perfect, the Steinberg character is primitive. 
  In the remaining cases, we take a degree $6$ character.
       \end{proof} 

  \section{Almost Simple Groups}
  Our objective in this Section is to prove Theorem \ref{ccl} for almost simple groups $G$ with socle $N$. We start with the following.

  %  
%  \begin{lem}\label{l:abelian}   Let $G$ be an almost simple group with socle $N$.  Let $H$ be a  Hall  subgroup  of
%  $G$ of odd index.  Assume that  $H \cap N$ is an elementary abelian $2$-group.   Then there exist $a, b \in H \cap N$ that are conjugate
%  in $G$ but not conjugate in $H$.
%  \end{lem}
%  
%  \begin{proof} Notice that $1<H \cap N$ is a Sylow $2$-subgroup of $N$, and therefore $N$ has abelian Sylow $2$-groups.   
%  Notice that $N<G$ by the Brauer-Suzuki Theorem. If $G/N$ is a $2$-group,  then $H$ is a Sylow $2$-subgroup of $G$. Let $1 \ne z \in \zent H \cap N$. 
%  Suppose that $z^G \cap H=z^H$. Now, if $z, z^g \in H$ for some $g \in G$, then $z^g=z^h=z$.
%  By Glauberman's  $Z^*$-theorem (Theorem 7.9 of \cite{N98}), we will have that $z \in \zent G$, a contradiction.
%  Therefore, $z^H \subset z^G \cap H$ and the lemma follows in this case.  
%  Since the simple group $N$ has abelian Sylow 2-subgroups, we have that  either $N=J_1$, $N \cong \mathrm{PSL}(2,q)$
%  with $q>3$,
%  or a Ree group $N=\mathrm{R}(3^e)$.    Let $m=|H \cap N|$. Note that $m-1$ divides $|N|$.  Since $H$ is a Hall subgroup, 
%  $\gcd(|H|, m-1) = 1$, $H$ cannot act transitively on a set of size $m-1$. 
%    \end{proof} 
%    
    \begin{lem}\label{l:abelian2}   Let $G$ be a finite group and $N \nor G$.   Let $H$ be a  Hall  subgroup  of
  $G$ and assume that  $Q=H \cap N$ is an abelian $q$-group.   Suppose that $N$ does not have a normal $q$-complement. Then there exists $a \in Q$ such that $a^G \cap Q \ne a^H$. 
  \end{lem}
  
  \begin{proof} Notice that $Q=H \cap N \in \syl qN$ using that $H$ is a Hall subgroup. Write $L=\norm GQ$.
    Suppose  that $a^G \cap Q=a^H$ for every $a \in Q$.  Fix $a \in Q$. If $l \in L$, then there is $h \in H$ such that
  $a^{lh^{-1}}=a$. Thus $L=H\cent La$.   Hence $|L: \cent La|$ is coprime with $\norm NQ/Q$. 
  Since $\norm NQ \nor L$, we have that $\norm NQ \sbs \cent La$.  Therefore $Q \sbs \zent{\norm NQ}$.
  By Burnside's theorem, we have that $N$ has a normal $q$-complement, but this is not possible.
  \end{proof}

  We first consider the case where $N$ is a sporadic group, with a slightly  more general hypothesis.

  \begin{lem} \label{l:sporadic} 
  Let $G$ be an almost simple group, with socle $N$,  a sporadic simple group.  
  Suppose that $H$ is a 
  subgroup of $G$ of odd index  and $H \cap N$ is  a minimal normal subgroup of $H$
  with $H<HN$.  Then  $G=N$ and either $H$ has odd prime order or 
  $G=M_{23}$ and $H=M_{22}$.  In both cases, $H$ is not $c$-closed in $G$.
  \end{lem}
   
   \begin{proof}  If $H \cap N$ is a $2$-group, then $H$ is a $2$-group and since $|H \cap N| \ge 4$, $H \cap N$ is not a minimal normal
   subgroup of $H$.  If $H \cap N$ is an elementary abelian $p$-group with $p$ odd, then $H \le N$ and so $H$ is not $c$-closed. 
      
   So we may assume that $L=H \cap N$ is a direct product of nonabelian simple groups.    Let $M$ be a maximal subgroup of $N$.  
  Then $\oh 2M \le \oh 2L =1$.   Since   $L$  is a perfect  Hall subgroup,  $M$ modulo
 the last term of its derived series has odd order.  
  
 The maximal subgroups of sporadic are all given in \cite{Wi} aside from the case of $N=Fi_{24}'$.  See \cite{LW} for this case.
  Aschbacher also studies the overgroups of Sylow $2$-subgroups  \cite{As}.  
  
  By inspection of the maximal subgroups of odd index in a sporadic simple group, the only cases where $\oh 2M =1$ and $M/[M,M]$
  has odd order are
  $N=M_{23}$ and $M=M_{22}$ or $N=McL$ and $M=M_{22}$ or ${\rm PSU}(4,3)$.    All the proper subgroups of odd index
  of $M_{22}$ and ${\mathrm PSU}(4.,3)$  have nontrivial $\bold{O}_2$ and so $H \cap N$ is maximal in $M$.  Checking the character tables 
  shows these subgroups are not $c$-closed. 
     \end{proof}

  We next consider symmetric and alternating groups.   
  
  \begin{lem}  \label{l:alt}  Let $N={\sf A}_n$ with $n > 4$.  Assume that $G=HN \le {\sf S}_n$ and $H$ is a proper subgroup of $G$
  containing a Sylow $2$-subgroup of $G$.
  \begin{enumerate}
  \item If  any two involutions of $H \cap {\sf A}_n$ conjugate in $G$ are conjugate in $H$, then 
  $H= {\sf A}_{n-1}$ or 
  ${\sf S}_{n-1}$ with $n$ odd. 
  \item If $H$ satisfies (i) and is a Hall subgroup, then $H= {\sf A}_{n-1}$ or 
  ${\sf S}_{n-1}$ with $n$ prime.  
  \item If   any two  ${\sf A}_n$ conjugate elements of $H \cap {\sf A}_n$ are $H$-conjugate, then $H = {\sf S}_{n-1}$.  
  \end{enumerate}
  \end{lem}
  
  \begin{proof} We first prove (i).   We induct on $n$.  If $n \le 8$, this follows by inspection.  First suppose that $H$ is not transitive.
  If $H$ has an invariant subset $\Gamma$ of size $d$ with $2 \le d  \le n/2$, then $H$ contains two involutions each 
  moving exactly $4$ points  -- one 
  trivial on $\Gamma$ and one nontrivial on $\Gamma$ and its complement  and so the result holds.
  
  Thus $H$ has precisely $2$ orbits of sizes $1$ and $n-1$ and so $n$ is odd.  Note that we can work in $G \cap {\sf S}_{n-1}$
  and by induction there are no examples.
  
  Suppose $H$ is transitive but not primitive and let $d$ be the size of a block.  If $d \ge 4$, there are two involutions in $H$
  moving exactly $4$ points with one non-trivial on two blocks and one trivial on all but one block and the result holds.
  If $d=2$, there are two such involutions with one fixing all blocks and another  switching two blocks.  If $d=3$, then since
  $n > 8$, we similarly find two involutions moving exactly $8$ points acting differently on blocks unless $n=9$.  In that case
  the stabilizer of the blocks does not contain a Sylow $2$-subgroup.
  
  Clearly (ii) follows.  Now we prove (iii).   By (i) we  may assume  $H={\sf A}_{n-1}$ with $n$ odd.  Let 
  $x \in H$ be an $n-2$ cycle.     Then there exists $x' \in H$ another $n-2$ cycle not conjugate to $x \in H$.  Since
  $x$ commutes with a transposition in ${\sf S}_n$, $x$ and $x'$ are conjugate in ${\sf A}_n$.
  \end{proof}

  In the one family above where $N={\sf A}_n$, $G=HN$, $H$ is a  subgroup of odd index and $H \cap N$ is a minimal normal subgroup of
  $H$, we need to prove the existence of a primitive irreducible character of $H \cap N$ that extends to $H$ but not to $N$.
  So $H = {\sf S}_{n-1}$ and $G={\sf S}_n$ with $n > 5$ odd.  The only prime  dividing $|G|$ but not $|H|$ is $n$ if $n$ is prime.   Aside
  from that case,  we can take the irreducible representation of $H \cap N$ of degree $n-2$.  In the case of characteristic $n$,
  we need to produce a different character.  We pick a particular one but in fact the argument below shows that there are many such.
  
  \begin{lem} \label{l:js}   Let $n > 5$ be a prime.  Let $\chi$ be the irreducible character of ${\sf A}_{n-1}$ corresponding to the partition
  $(n-3, 2)$.  Then $\chi$ is primitive,  extends to a
  complex character of $\mathrm{Aut}({\sf A}_{n-1})$ and does not extend to an irreducible $n$-Brauer character of
  ${\sf A}_n$.  
 \end{lem}
 
 \begin{proof}  Clearly $\chi$ extends to $\mathrm{Aut}(G)$ (if $n=7$, then  one checks the modular character table).   By considering
 the degree, it is clear it is primitive. 
  Suppose $\chi$ extends to an irreducible module $D^{\lambda}$ in characteristic $n$ for $N$.   By Kleshchev's branching rules
  or by a result of James \cite[Theorem C]{Ja} the only possibility is that $\lambda = (n-2,2)$.  By the solution of the Jantzen-Seitz
  conjecture \cite{BK},  $D^{(n-2, 2)}$ does not restrict to an irreducible ${\sf S}_{n-1}$ module.
 \end{proof}

  \begin{lem}\label{lie2}
   Let $G=HN$ be an almost simple group with socle $N$, a finite group of Lie type over the field of $q$ elements and $H$ a proper  subgroup of
  $G$ of odd index..   Assume that $R=H \cap N$ is a minimal normal subgroup of $H$.  Then one of the following holds:
  \begin{enumerate}
  \item   $N=\mathrm{PSL}(2,q)$,  $q \equiv \pm 1 \pmod 5$,  $q \equiv \pm 3 \pmod 8$ and $R= {\sf A}_5$; or
  \item   $N=\Omega(7,q)$ with $q \equiv \pm 3 \pmod 8$ and $R =\Omega(7,2)$; or 
   \item   $N=\Omega^+(8,q)$ with $q \equiv \pm 3 \pmod 8$  and $R =\Omega^+(8,2)$.
   \item   $R=G(q_0)$ where $q$ is odd and  $q=q_0^c$ for some odd $c$.
   \end{enumerate} 
   In all cases, there exists a primitive irreducible character $\gamma$ of $R$ that extends to a complex character of
    $\mathrm{Aut}(R)$ and does  not extend to an irreducible complex character of  $N$.
   \end{lem} 
   
   \begin{proof}   We use the classification of maximal subgroups of odd index \cite{LSa}.   Let $M$ be a maximal subgroup
   of $N$ containing $R$.  Since $O_2(R)=1$ amd $R/[R,R]$ has odd order, the same is true for $M$.  By the main result of
   \cite{LSa}, the only possibilites are for $R$ are as listed above.   If $R$ is proper in $M$, then we must be in the last case
   the result follows by induction. 
   
    The last fact follows  in the first three cases by taking  $\gamma$ of degree $4, 7$ or $28$ respectively.  
  These do extend to the automorphism group and by \cite{SZ} are too small to extend to $N$.
  
  In the last case aside from the possibility that $R=PSL(2,7)$ or $PSp(4,3)$,
 let $\chi$ be the Steinberg character0 of $R$.  Then $\chi$ is primitive
  as in the proof of Lemma \ref{prim}.  By the main result of \cite{MZ}, $\chi$ does not extend to $N$.  If the remaining cases, let
  $\chi$ be the unique irreducible character of degree $6$.
  \end{proof}

   \begin{lem}\label{lie}
   Let $G=HN$ be an almost simple group with socle $N$, a finite group of Lie type over the field of $q$ elements and $H$ a proper Hall subgroup of
  $G$ of odd index.   Let $p$ be a prime not dividing $|H|$.  Assume that $R=H \cap N$ is a minimal normal subgroup of $H$.  Then one of the following holds:
  \begin{enumerate}
  \item   $N=\mathrm{PSL}(2,q)$,  $q \equiv \pm 1 \pmod 5$,  $q \equiv \pm 3 \pmod 8$ and $R= {\sf A}_5$; or
  \item   $N=\Omega(7,q)$ with $q \equiv \pm 3 \pmod 8$ and $R =\Omega(7,2)$; or 
   \item   $N=\Omega^+(8,q)$ with $q \equiv \pm 3 \pmod 8$  and $R =\Omega^+(8,2)$.
   \end{enumerate} 
   In all cases, there exists a primitive irreducible character character $\gamma$ of $R$ that extends to a complex character of
    $\mathrm{Aut}(R)$ and does  not extend to an irreducible $p$-Brauer  character of  $N$.
   \end{lem}

   \begin{proof} This follows immediately from the previous lemma noting that in the last case there $H$ is not a Hall subgroup.
   The only additional assertion is in the case of $p$-Brauer characters.  Then we take $\gamma$ of degree $4, 7$ or $28$ respectively.  
  These do extend to the automorphism group and by \cite{SZ} are too small to extend to $N$ in any
  characteristic $p$ unless possibly $p$ divides $q$.  In that case, take the characters of degree
  $4, 15$ or  $50$, respectively and use \cite{Lu}. 
  \end{proof}

  \begin{thm}\label{ccl}
 Let $p$ be a prime.
  Suppose that $G$ is an almost simple group with socle $N$ and a Hall subgroup  $H$
  of order not divisible by $p$ such that $H<HN$. 
   Assume that $R=H \cap N$ is a minimal normal subgroup of $H$. 
Assume that $r^G \cap R=r^H$ for every $r \in R$.  
  Then $R$ is simple non-abelian and has a primitive character $1 \ne \gamma \in \irr R$ that extends to a
  complex character of ${\rm Aut}(R)$ but not to 
  an irreducible $p$-Brauer character of $N$.
  \end{thm}

  \begin{proof}
  We may assume that $G=NH$.   By Lemma \ref{l:abelian2},  $N \cap H$ is a direct product of nonabelian simple groups.
  In particular $H$ contains a Sylow $2$-subgroup of $G$.
  The case where $N$ is of Lie type
  follows from Lemma \ref{lie}.   If $N$ is sporadic, this follows by Lemma \ref{l:sporadic}.  If $N$
  is alternating, the result follows by Lemmas \ref{l:alt} and \ref{l:js}. 
  \end{proof}
  
 Essentially the same proof gives:
  
  \begin{thm}\label{ccl2}
  Suppose that $G$ is an almost simple group with socle $N$ and a subgroup $H$
  of odd index such that $H<HN$.   Then there exists  a character $1 \ne \gamma \in \irr H$
   that does not extend to 
  an irreducible  character of $G$.
  \end{thm}

  \begin{proof}  We may assume that $G=HN$.  Suppose that every irreducible character of $H$ extends.  By Lemma \ref{elem},   $H$ is $c$-closed in
  $G$ and $H \cap N$ is a minimal normal subgroup of $H$.  Now argue precisely as in the previous proof.
  \end{proof}
  
  \section{Main result}
     \begin{thm}\label{main}
 Let $G$ be a finite group,  let $H$ be a  Hall subgroup of $G$
 which is $c$-closed in $G$, with $|G:H|$ odd.
 Assume that for every $L \nor H$ there exists $K \nor G$ such that $K\cap H=L$. Let $p$ be a prime not dividing $|H|$.
 If every primitive character of $H$ extends to
 some $p$-Brauer character of  $G$, then $H$ has a normal complement in $G$.
 \end{thm}
 
 \begin{proof}
 Let $G$ be a  counterexample minimizing $|G|$. Let $\pi$ be the set of divisors of $|H|$. 
 We may assume that $H<G$.  Let $N$ be a minimal normal subgroup of $G$.
 If $N \sbs H$, then $H/N$ is $c$-closed in $G/N$ by Lemma 2.2(c) in \cite{S}. By minimality, we have that
 $H/N$ has a normal complement $K/N$ in $G/N$. 
 Then we apply Lemma \ref{nor}.  Hence, we may assume that $N$ is not contained in $H$.
 
 Suppose first that $N$ is an abelian $q$-group. 
 Then $q$ does not divide $|H|$, and therefore $H \cap N=1$.
 We claim that $HN/N$ satisfies the hypothesis
 of the theorem.  
 We have that $HN/N$ is $c$-closed in $G/N$ (for instance, by Lemma 2.2(d) in \cite{S}).  Suppose that $R/N$ is normal in $HN/N$. Let $L=R\cap H \nor H$. By the hypothesis, there is $K \nor G$ such that
 $K\cap H=L$.  Since $KN/K$ is $\pi'$, then $H \cap KN=H\cap K=L$. 
 Then $KN \cap HN=LN=R$.
 Finally, let $\alpha \in \irr{HN/N}$ be primitive. Then
 $\alpha_H$ is primitive, and therefore it has an extension to $G$ by hypothesis.  By  Lemma \ref{prim}, there is an extension $\beta \in \ibr{G}$ of $\alpha_H$ that contains $N$ in its kernel. Necessarily we have that $\alpha=\beta_{HN}$. We have that $HN/N$ satisfies the hypothesis, as claimed, and by minimality we are done. 
 
 So we may assume that $N$ is a direct product of non-abelian simple groups.
 Let $W=HN$. We have that $H$ is c-closed in $W$ and that every primitive character of $H$ extends to some $p$-Brauer character of $W$.
 Also, if $L \nor H$, then there is $K \nor G$ such that $K\cap H=L$,
 and therefore $K \cap W \cap H=L$.
 Suppose that $W<G$. By minimality, we have that $H$ has a normal
 complement $F_0$ in $W$. Now $N$ has $\pi$-index in $W$, so it contains $F_0$. But $F_0 \nor N$ so it is a direct product of some of the factors of $N$.
 However, $|F_0|=|W:H|$ is odd. This is a contradiction.
 
 So we may assume that $G=HN$, where $N$ is any minimal normal subgroup of $G$. Since $G/N$ is a $\pi$-group, we have that $N$ is the unique minimal normal subgroup of $G$ (since otherwise $G \le G/N \times G/M$ would be a $\pi$-group).
 
Using that every normal subgroup of $H$ is the intersection of a normal subgroup of $G$ with $H$, we deduce that $H \cap N$ is a minimal normal subgroup of $H$. In fact, it is the unique one. 

Let $S \nor N$ be simple. Then $N \le \norm GS$. If $H=\norm HS h_1 \cup \ldots \cup \norm HS h_t$, then notice that $N=S^{h_1}\times  \cdots \times S^{h_t}$, using that $N$ is a minimal normal subgroup of $G$.
Also, $H\cap N=(H \cap S)^{h_1} \times \cdots \times (H \cap S)^{h_t}$,
using that $H \cap N$ is a minimal normal subgroup of $H$.
Since $S \nor\nor G$, we have that $S \cap H$ is a Hall $\pi$-subgroup of $S$. Also, we have that $S \cap H$ is a minimal normal subgroup of
$\norm HS$.  Write $J=S \cap H$. 
We claim that $x^{\norm GS} \cap J=x^{\norm HS}$ for every $x \in J$. Suppose that $x, x^v \in J$, where $v \in \norm GS$, and $x \ne 1$. By hypothesis, we have that $x^v=x^h$ for some $h \in H$.
Then $1 \ne x \in S \cap S^{h^{-1}}$, and therefore $S=S^h$. Hence $x^v=x^h$ for
some $h \in \norm HS$, as claimed.

Now, we claim that the same property holds in $JC/C$ with respect to $\norm GS/C$, where $C=\cent GS$.
Let $x\in J$ and suppose that $x^v \in JC$, where $v \in \norm GS=N\norm HS$. Write $N=T \times S$ and $v=wst$, where $w \in \norm HS$, $s \in S$ and $t \in T$. Then $x^v=(x^{ws})^t$. Since $J \nor \norm HS$, we have that
$x^v=x^{ws}$. Write $x^{ws}=jc$, where $j \in J$ and $c \in C$. 
Hence $j^{-1}x^{ws}=c \in S \cap C=1$. Therefore $x^{v} \in J$. But then
$x^v=x^k$ for some $k \in \norm HS$. This proves the claim. 

We claim now that  the almost simple group $\norm GS/C$ satisfies the   hypothesis of Theorem \ref{ccl} with respect to the socle $SC/C$ and the Hall $\pi$-subgroup $\norm HSC/C$. Notice first that the socle $SC/C$
is not a $\pi$-group. Otherwise,   $S$ is a $\pi$-group, $N$ is a $\pi$-group
  and $H=G$, which is not possible.
Let $R=(\norm HSC/C) \cap (SC/C)$. We need to prove that $R$ is a minimal normal subgroup
of $\norm HSC/C$ and that
$$r^{\norm GS/C} \cap R=r^{\norm HSC/C}$$
for $r \in R$. 
Notice now that $J=H \cap S$ is a Hall $\pi$-subgroup of $S\nor\nor G$.
Then $JC/C$ is Hall $\pi$-subgroup of $SC/C$ (since the natural map $S \rightarrow SC/C$
is an isomorphism).  Also $\norm HS$ is a Hall $\pi$-subgroup of $\norm GS$ and therefore $\norm HS \cap SC$ is a Hall $\pi$-subgroup of $SC$. Thus $(\norm HS \cap SC)C/C$ is a Hall $\pi$-subgroup of $SC/C$.  Therefore
$$(\norm HS \cap SC)C/C=(H\cap S)C/C =JC/C\, .$$
By Dedekind's lemma, $(\norm HS \cap SC)C=\norm HSC \cap SC$. 
And we conclude that $R=JC/C$, which is naturally isomorphic to $J$.
Since $J$ is a minimal normal subgroup of $\norm HS$, we conclude that $R$ is a minimal normal subgroup
of $\norm HSC/C$.  The equality
$$r^{\norm GS/C} \cap R=r^{\norm HSC/C}$$
follows now from the previous paragraph.

By Theorem \ref{ccl}, we have that
$R$ is simple non-abelian    and has a primitive character $1 \ne \gamma \in \irr R$ that extends to a character of ${\rm Aut}(R)$ but not to a $p$-Brauer character of  $SC/C$. 
Using the natural isomorphism $S \rightarrow SC/C$, this easily implies that $J=H\cap S$ is a non-abelian simple group,
and $J$ has a primitive character $1 \ne \gamma \in \irr J$ that extends to 
some  character of ${\rm Aut}(J)$ but not to a $p$-Brauer character of $S$.
By Theorem \ref{pext}, there is a primitive $\alpha \in \irr H$ such that $\alpha_{H\cap S}=\gamma$. By hypothesis, there is $\tau \in \ibr G$ such that $\tau_H=\alpha$. Then $\tau_{H\cap S}=\gamma$, and therefore $\tau_S$ extends $\gamma$. This is a contradiction.
\end{proof}

We are now ready to prove Theorem A.
\begin{cor}
Let $G$ be a finite group and let $H$ be a Hall subgroup of $G$.  Let $p$ be a prime not dividing $|H|$.
 If every irreducible character of $H$ extends to
 some $p$-Brauer character of  $G$, then $H$ has a normal complement in $G$.
 \end{cor}

\begin{proof}
By Theorem \ref{solv}, we may assume that $H$ is non-solvable. Hence, $|G:H|$ is odd. 
Now we apply Lemma \ref{elem} and Theorem \ref{main}.
\end{proof}

Notice that the conclusion of Theorem \ref{main} does not hold if the hypothesis that every normal subgroup $L$ of $H$ is the intersection
of a normal subgroup of $G$ with $H$ is dropped, as shown by ${\sf S}_4$ and ${\sf S}_5$.  

\medskip

%We also remark that Theorem A was proven in \cite{F1} under the additional hypothesis that every $\pi$-element of $G$ lies in a $G$-conjugate of $H$. 
%%
\medskip

%Notice too that the proof of Theorem \ref{main} shows the following result
%on solvable groups that seems to be new.
%
%\begin{thm}\label{solv}
% Let $G$ be a finite solvable group,  let $H$ be a  Hall subgroup of $G$
% which is $c$-closed in $G$.
% Assume that for every $L \nor H$ there exists $K \nor G$ such that $K\cap H=L$. 
% If every primitive character of $H$ extends to $G$, then $H$ has a normal complement in $G$.
% \end{thm}
 
% There are many examples of pairs $(G,H)$ where $H \le G$, $H$ does not have a normal complement in 
% $G$, and every irreducible character of $H$ extends to $G$. For instance, if $H \nor G$, $G/H=\langle aH\rangle$
% and $a$ fixes all conjugacy classes of $H$. A more interesting example is the following:
% Let $L$ be a $2$-transitive Frobenius group $AD$ with $A$ the normal subgroup of $L$
%and $D$  the Frobenius complement with $|A|=q$ say.
%Then $L$ embeds in ${\sf S}_q$. 
%Let $G = {\sf S}_q  \times D$ and take $H = AC$ where $C$
%is the diagonal of
%$D \times D$.
%Then all irreducible chararacters of $H$ extend to $G$. 


\medskip
  
    
\begin{thebibliography}{ABCDEFGH}
 
\medskip

\bibitem[A]{As}  M. Aschbacher,   Overgroups of Sylow subgroups in sporadic groups,
 Mem. Amer. Math. Soc. {\bf 60} (1986), no. 343, iv+235 pp.

\bibitem[B]{B} R. Baer, Closure and dispersion of finite groups, Illinois J. Math. {\bf 2} (1958), 619--640.

\bibitem[BCLP]{BCLP} M. Bianchi, D. Chillag, M. L. Lewis and E. Pacifici,
Character degree graphs that are complete graphs, Proc. AMS. {\bf 135}
(2005), 671--676.

\bibitem[BK]{BK}  J. Brundan and A. Kleshchev,  Representations of the symmetric group which are irreducible over subgroups, 
J. Reine Angew. Math. {\bf 530}  (2001), 145--190.

%\bibitem[F1]{F1} P. Ferguson, Relative normal complements and extendibility of characters, Arch. Math. {\bf 42} (1984), 121--125.

\bibitem[F]{F} P. Ferguson,  
Character restriction and relative normal complements.
J. Algebra {\bf 120} (1989), no. 1, 47--53.

\bibitem[Fr]{Fr}  F. G. Frobenius,  \"Uber aufl\"osbare Gruppen, IV, V, Sitzungber. Akad. Wiss. Berlin, (1901), 1216--1230, 1324--1329.

\bibitem[G]{G} R. M. Guralnick, Subgroups of prime power index in a simple group, 
J. Algebra {\bf 81} (1983),   304--311.

\bibitem[H]{H}
N. S. Hekster,  
On finite groups all of whose irreducible complex characters are primitive,
  Indag. Math. {\bf 47} (1985), no. 1, 63--76.
  
  
  \bibitem[Hi]{Hi}
  D. G. Higman, Focal series in finite groups, Canad. J. Math. {\bf 5} (1953),
  477--497.
  
%  \bibitem[I1]{I1}
% I. M. Isaacs, Subgroups with the character restriction property,
% J. Algebra {\bf 100} (1986), 403--420. 
 
  \bibitem[I]{I}
I. M. Isaacs, `{\it Character Theory of Finite Groups}',
AMS-Chelsea, Providence, 2006.

\bibitem[Ja]{Ja}  G.  James, On the decomposition matrices of the symmetric groups, II, J. Algebra {\bf 43} (1976), 45--54.

\bibitem[LS]{LSa} M. W. Liebeck and J. Saxl,  The primitive permutation groups odd degree, J. London Math. Soc. (2) {\bf 31} (1985) 250--264.

\bibitem[LW]{LW}  S. A. Linton and R. W. Wilson, The maximal subgroups of the Fischer groups $Fi_{24}$ and $Fi_{24}'$, 
Proc. London Math. Soc. (3) {\bf 63} (1991),  113--164.

\bibitem[Lu]{Lu}  F. L\"ubeck,  Small degree representations of finite Chevalley groups in defining characteristic. LMS J. Comput. Math. {\bf 4} (2001), 135--169.
   
   \bibitem[M]{M} A. Moret\'o, Complex group algebras of finite groups: Brauer's Problem 1, Adv. Math. {\bf 2007}, 238--248.
   
   
 \bibitem[MN]{MN} G. Malle and G. Navarro, Extending Characters from Hall Subgroups,
Doc. Mathematica {\bf 16} (2011),  901--919. 

\bibitem[MZ]{MZ} G. Malle and A. E. Zalesskii,   Prime power degree representations of quasi-simple groups,  Arch. Math. {bf 77} (2001) 461--468.
   
   \bibitem[N1]{N1} G. Navarro, `{\it Blocks and Characters of Finite Groups}', Cambridge University Press, 1998.
   
\bibitem[N2]{N2} G. Navarro, `{\it Character Theory and the McKay Conjecture}', Cambridge University Press, 2018.

%% \bibitem[RV]{RV}   D.  O.  Revin and E. P.  Vdovin, Hall subgroups of finite groups, Ischia group theory 2004, 229--263, Contemp. Math. {\bf 402}, Israel Math. Conf. Proc. 
%% Amer. Math. Soc., Providence, RI, 2006.  
 
 \bibitem[S]{S} C. H. Sah, Existence of normal complements and extensions of characters in finite groups, Illinois J. Math. {\bf 6} (1962), 282--291. 
 
 \bibitem[Se]{Se} G. M. Seitz,  Flag-transitive subgroups of Chevalley Groups, 
Ann. of Math.   {\bf} 97 (1973), 27--56. 

\bibitem[SZ]{SZ}  G. M. Seitz and A. E. Zalesski,  On the minimal degrees of projective representations of the finite Chevalley groups. II
J. Algebra {\bf 158} (1993),  233--243.
 
 \bibitem[T]{T} J. G. Thompson, Normal $p$-complements for finite groups,
 Math. Z. {\bf 72}, (1960), 332--354.
 
 \bibitem[W]{W} H. Wielandt, \"Uber die Existenz von Normalteilern in Endlichen Gruppen, Math. Nachr. {\bf 18} (1958), 274--280.
 
 \bibitem[Wi]{Wi} R. A. Wilson, https://brauer.maths.qmul.ac.uk/Atlas/v3/spor/
 
\end{thebibliography}
\end{document}